\documentclass[draft, reqno]{amsart}
\usepackage[all]{xy}                        %

\CompileMatrices                            

\UseTips                                    

\input xypic
\usepackage[bookmarks=true]{hyperref}       

\usepackage{amssymb,latexsym,amsmath,amscd}
\usepackage{xspace}
\usepackage{enumerate}
\usepackage{graphicx}
\usepackage{vmargin}
\usepackage{todonotes}

\reversemarginpar
\theoremstyle{plain}
\newtheorem{theorem}{Theorem}[section]
\newtheorem*{theorem*}{Theorem}
\newtheorem{proposition}[theorem]{Proposition}
\newtheorem{corollary}[theorem]{Corollary}
\newtheorem{lemma}[theorem]{Lemma}

\newtheorem{question}[theorem]{Question}

\theoremstyle{definition}

\newcommand{\enm}[1]{\ensuremath{#1}}          %

\newcommand{\cal}[1]{\mathcal{#1}}

\newcommand{\PP}{\enm{\mathbb{P}}}

\newcommand{\Cc}{\enm{\cal{C}}}
\newcommand{\Dd}{\enm{\cal{D}}}

\newcommand{\Ll}{\enm{\cal{L}}}

\renewcommand{\phi}{\varphi}
\renewcommand{\theta}{\vartheta}
\renewcommand{\epsilon}{\varepsilon}

\begin{document}

\title{Birational geometry of defective varieties, II}

\thanks{The authors were partially supported by GNSAGA of INdAM and by PRIN 2017 ``Moduli Theory and Birational Classification".} 

\author[E. Ballico]{Edoardo Ballico}
\address{Dipartimento di Matematica, Universit\`a degli Studi di Trento\\
Via Sommarive 14, 38123 Povo, Italy}
\email{edoardo.ballico@unitn.it}

\author[C. Fontanari]{Claudio Fontanari}
\address{Dipartimento di Matematica, Universit\`a degli Studi di Trento\\
Via Sommarive 14, 38123 Povo, Italy}
\email{claudio.fontanari@unitn.it}



\date{}

\begin{abstract}
Let $X \subset \PP^r$ be smooth and irreducible and for $k \ge 0$ let $\nu_k(X)$ 
(resp., $\delta_k(X)$) be the $k$-th contact (resp., the $k$-th secant) defect of $X$. 
For all $k \ge 0$ we have the inequality $\nu_k(X) \ge \delta_k(X)$ and 
in the case $k=1$ we characterize projective varieties $X$ for which 
equality holds, $\dim \mathrm{Sing}(X) \le \delta _1(X) -1$ and the 
generic tangential contact locus is reducible. 
\end{abstract}

\maketitle

\section{Introduction}

This short note is a follow-up to our joint paper \cite{bf}, going back to 2004. 

Let $X \subset \PP^r$ be a reduced and irreducible projective variety of dimension $n$.
Fix an integer $k \ge 0$ and assume that $\mathrm{Sec}_k(X) \subset \PP^r$.

Let $\nu_k(X)$ be the dimension of the contact locus of a general hyperplane tangent at 
general points $p_0, \ldots, p_k \in X$ (see \cite[end of p. 152]{cc}). 
Since a general hyperplane tangent at $p_0, \ldots, p_k$ is a special hyperplane tangent at 
$p_0, \ldots, p_{k-1}$, by semicontinuity we have 
\begin{equation}\label{contact}
\nu_{k}(X) \ge \nu_{k-1}(X).
\end{equation}
Let $\delta_k(X)$ be the integer 
$$
\delta_k(X) := s_{k-1}(X)+n+1-s_{k}(X),
$$
where $s_{k}(X)$ denotes the dimension of the $k$-secant variety $\mathrm{Sec}_k(X)$ of $X$ 
(see \cite[(1.5.4)]{z} and \cite[Definition 1.4.6]{r2}).   

On the other hand, let $\delta^k(X)$ be the integer
$$
\delta^k(X) := \min \{r, n(k+1)+k \} - \dim \mathrm{Sec}_k(X)
$$
(see \cite[p. 152]{cc} and \cite[p. 13]{bf}).

By definition, we have 
\begin{equation}\label{defect}
\delta^k(X) = \delta_0(X) + \ldots + \delta_k(X),
\end{equation}
in particular $\delta^k(X) = \delta_k(X)$ if $k$ is the minimal integer such that $X$ is $k$-defective.

In \cite[Proposition 1]{bf} we pointed out the following inequality:

\begin{proposition}
Assume that $k$ is the minimal integer such that $X$ is $k$-defective. Then we have 
$\nu_k(X) \ge \delta^k(X) = \delta_k(X)$.
\end{proposition}

Indeed, the above inequality holds for every integer $k$: 

\begin{proposition}\label{c1}
We have $\nu_k(X) \ge \delta_k(X)$ for all $k \ge 0$.
\end{proposition}

As an immediate consequence, a shorter proof of Terracini's theorem follows 
\cite[Theorem 1.1]{cc} (see Corollary \ref{c2}). 

Next, we focus the following conjecture, stated at the end of the Introduction of \cite{bf} 
and addressing the case in which equality holds:
\emph{we suspect that a defective variety with $\nu_k(X) = \delta_k(X)$ and reducible contact 
locus should be a cone}.

Here we are able to confirm it for $k=1$ under the assumption $\dim \mathrm{Sing}(X) \le \delta _1(X) -1$:

\begin{proposition}\label{i2}
Let $X$ be a non-degenerate variety $X\subset \PP^r$, $r\ge 2n+2$, such that $\nu _1(X) =\delta _1(X) >0$, 
$\dim \mathrm{Sing}(X) \le \delta _1(X) -1$ 
and the generic tangential contact locus is reducible. Then $\mathrm{Sing}(X) =\delta _1(X) -1$ and $X$ is a cone with $(\delta_1(X)-1)$-dimensional vertex.
\end{proposition}

Finally, we ask the following: 

\begin{question}\label{p2}
Under which assumption on $X$ the inequality $\nu_{k+1}(X) - \nu_k(X) \ge \delta_{k+1}(X) - \delta_k(X)$ 
holds for all $k$ such that $\mathrm{Sec}_{k+1}(X) \ne \PP^r$?
\end{question}

Indeed, we believe that the above inequality cannot hold without a suitable assumption: a potential source 
of counterexamples is provided by the following: 

\begin{proposition}\label{t2}
Fix an integer $k > 0$. Let $X\subset \PP^r$ be an integral and non-degenerate variety such that 
$\mathrm{Sec}_{k+1}(X) \subset \PP^r$. Assume $\nu_{k-1}(X) =\nu_k(X) >0$, $\delta_k(X)=0$ and 
$\delta_{k+1}(X)>0$. Then the general tangential contact locus of order $t\in \{k-1,k \}$ 
is of type II (i.e. reducible) and for a general $S\subset X$, $\sharp (S)=t+1$, there is a union 
of $t+1$ $\nu_k(X)$-dimensional linear spaces $L_p\subset X$, $p\in S$, such that $p\in L_p$ and 
for $t\in \{k-1, k \}$ a general union of $t+1$ $L_p$'s is linearly independent and gives the 
$t$-tangential contact locus. One of the following cases occurs:
\begin{enumerate}
\item $\nu_{k+1}(X)=\nu_k(X)$, a general union of $k+2$ $L_p$'s gives the $(k+1)$-tangential 
contact locus and it spans a linear space of dimension $(k+2)(\nu_k(X)+1)-1-\delta_{k+1}(X)$;
\item $\nu_{k+1}(X) > \nu_k(X)$ and for a general $p\in X$ there is an integral 
$\nu_{k+1}(X)$-dimensional variety $A_p\subset X$ such that a general union of 
$k+2$ $A_p$'s is the general $(k+1)$-tangential contact locus.
\end{enumerate}
Moreover, in case (ii) we have $\nu_k(X)\le \dim X-2$, a general join of $t\le k+1$ $A_p$'s 
has dimension $t(\nu_{k+1}(X)+1)-1$, while the join of $k+2$ general $A_p$'s is a linear space 
of dimension $(k+2)(\nu_{k+1}(X)+1)-1 -\delta_{k+1}(X)$.
\end{proposition}

We work over an algebraically closed field of characteristic $0$.

We are grateful to the anonymous referee for detailed and appropriate comments to a previous version of this paper. 

\section{The proofs}

\begin{proof}[Proof of Proposition \ref{c1}:]
For a general element $q$ of the $k$-secant variety $\mathrm{Sec}_k(X)$ of $X$, its entry locus 
(in the sense of \cite[Definition 1.4.5]{r2}) $\Sigma_q(X)$ has pure dimension $\delta_k(X)$. 
The tangent space to $\mathrm{Sec}_k(X)$ is tangent to $X$ at each point of $\Sigma_k(X)$ 
(\cite[part 1. of Corollary 1.4.7]{r2}).

\end{proof}

The following is Terracini's theorem \cite[Theorem 1.1]{cc}:

\begin{corollary}\label{c2}
We have $(k+1)\nu_{k}(X) \ge \delta^k(X)$.
\end{corollary}

\begin{proof}
By (\ref{contact}), Proposition \ref{c1}, and (\ref{defect}) we have 
\begin{eqnarray*}
(k+1) \nu_k(X) &\ge& \nu_0(X) + \ldots + \nu_k(X) \ge \\
&\ge& \delta_0(X) + \ldots + \delta_k(X) = \delta^k(X). 
\end{eqnarray*}
\end{proof}

In order to prove Proposition \ref{i2} we need the following:

\begin{lemma}\label{i1}
Fix integers $k\ge 1$, $n\ge 2$ and $r\ge (k+1)(n+1)$. Let $X\subset \PP^r$ be an integral and non-degenerate $n$-dimensional variety, which is $k$-defective, but not $(k-1)$-defective; if $k\ge 2$ assume that $X$ is not $(k-1)$-weakly defective. Assume $\nu _k(X)=\delta _k(X)$ and that the general tangential contact locus is reducible. Then each irreducible component of the general tangential contact locus is a linear space.
\end{lemma}

\begin{proof} Fix a general $(p_0,\dots ,p_k)\in X^{k+1}$ and let $\Sigma$ 
(or $\Sigma (p_0,\dots p_k)$) denote the tangential contact locus of $X$ at $\{p_0,\dots ,p_k\}$.
By assumption $\Sigma$ has $k+1$ irreducible components  $\Sigma _i$, $0\le i \le k$, 
(or $\Sigma _i(p_0,\dots p_k)$) with $p_j\in \Sigma _i$ if and only if
$i=j$ and $\dim \Sigma _i =\delta _k(X)$ for all $i$. If we want to stress the dependency on $X$ of the contact locus we write $\Sigma _X$ and $\Sigma _{i,X}$ instead of $\Sigma$ and $\Sigma _i$.

\quad (a) Assume $k=1$ and $\delta _1(X)=1$. Since $\nu _1(X) =1$, the so-called Terracini's theorem 
(\cite[Theorem 1.1]{cc}) gives $1\le \dim \langle \Sigma \rangle \le 2$. If $\Sigma$ were a line, 
then $X$ would be a linear subspace of $\PP^r$, contrary to the assumption. Thus $\Sigma$ is a plane curve. 
Since $\Sigma$ contains two general points, $p_0$ and $p_1$ of $X$, and the scheme-theoretic intersection 
of $X$ with one of its general secant lines is just $2$ points 
(\cite[Proposition 2.6]{cc}), we get $\deg (\Sigma )=2$, hence $\Sigma _0(p_0,p_1)$ and
$\Sigma _1(p_0,p_1)$ are different lines contained in a plane. Thus $\Sigma _0(p_0,p_1)$ and 
$\Sigma _1(p_0,p_1)$ meet at a single point $q_{p_0,p_1}$. 

\quad (b)  Assume $k=1$ and $\delta _1(X)\ge 2$. We do induction on $n$ and $\delta _1$. Let
$H\subset \PP^r$ be a general hyperplane. Since $r\ge (k+1)(n+1)-1$, $\{p_0,\dots ,p_k\}$ does not span 
$\PP^r$. Since $(p_0,\dots ,p_k)$ is general, we may assume $H\supset \{p_0,\dots ,p_k\}$.
Set $Y:= H\cap H$. Since $n\ge 2$, $Y$ is an integral $(n-1)$-dimensional variety spanning $H$. Since
we may take as $H$ a general hyperplane containing $\{p_0,p_1\}$, it is easy to check that 
$\Sigma_Y(p_0,p_1) = \Sigma(p_0,p_1)\cap H$. Thus $\nu _1(Y) =\nu (X)-1$. In the same way we
see that $Y$ is not weakly $(k-1)$-defective. By applying Terracini's lemma it is easy to check that
$\delta _1(Y) =\delta (X)-1$. Thus by induction on $n$ or $\delta_1(Y)$ we see
that $\Sigma_0\cap H$ and $\Sigma_1\cap H$ are linear spaces. 
Since $H$ is general, we get that $\Sigma_0$ and $\Sigma_1$ are linear spaces. By inductive assumption 
$\Sigma_1\cap H$ and $\Sigma_2\cap H$ are distinct $(\delta_1(X)-1)$-dimensional linear spaces meeting 
in a $(\delta_1(X)-2)$-dimensional linear subspace. By the generality of $H$, $\Sigma _0\cap \Sigma _1$ 
is a $(\delta _1(X)-1)$-dimensional linear space, i.e. $\dim \langle \Sigma \rangle = \delta _k(X)+1$. 

\quad ({c}) Assume $k\ge 2$. Since $X$ is not weakly $(k-1)$-defective and $r\ge k(n+1)$, the $k$-tangential 
projection $\tau_{X,k-1}$ is birational onto its image (\cite[Lemma 1]{bf}). We have $\nu_k(X) = 
\nu_1(X_{k-1})$ and $\delta _k(X) = \delta _1(X_{k-1})$ (\cite[eq (1) and (2)]{bf}). By taking as 
$\tau_{X,k-1}$ the linear projection from the linear space $\langle\cup _{i=2}^{k}T_{p_i}X\rangle$ we obtain 
the following more precise statement. Take $q_i:= \tau _{X,k-1}(p_i)$, $0\le i \le 1$, as two general
points of $X_{k-1}$ and taking $(p_0,p_1)$ general in $X^2$ after fixing $p_2,\dots ,p_k$. 
Since $\tau _{X,k-1}$ is birational onto its image, $\Sigma_{X_{k-1}}$ is reducible with two irreducible components, $\Sigma _{X_{k-1},1}$ and $\Sigma_{X_{k-1},2}$, which are the images of $\Sigma_0$ and 
$\Sigma_1$. Since $\delta _1(X_{k-1}) = \nu _1(Y)$ by steps (a) and (b) $\Sigma _{X_{k-1},0}$ and 
$\Sigma_{X_{k-1},1}$ are linear spaces of dimension $\delta_k(X)-1$ meeting in a 
$(\delta_k(X)-1)$-dimensional linear space, i.e. contained in a $(\delta _k(X)+1)$ linear subspace. 
Set $W:= \langle \cup _{i=2}{k} T_{p_i}X$. Since $X$ is not $(k-2)$-defective, we have 
$\dim W= (k-2)(n+1) -1$. Let $\ell _W: \PP^r\to \PP^{r-(k-1)(n+1)}$ denote the linear projection
from $W$ ($X_{k-1}$ is the closure of $\ell _W(X\setminus X\cap W)$ in $\PP^{r-(k-1)(n+1)}$. 
Since $X$ is not $(k-1)$-defective, $T_{p_1}X\cap W=\emptyset$. Hence $\ell _{W|T_{P_1}X}$ maps isomorphically onto it image, the tangent space to $X_{k-1}$ at $q_1:= \ell _w(p_1)$. 
Since $\Sigma _1\subset T_{p_1}X$ and $\Sigma _{1X_{k-1}}$ is a linear space, $\Sigma_1$ is a linear space. For the same reason $\Sigma _0$ is a linear space.

\end{proof}

\begin{proof}[Proof of Proposition \ref{i2}:]
We follow the set-up of the proof of Lemma \ref{i1} for the case $k=1$, i.e. 
the notation and steps (a) and (b) of that proof.

\quad (a) Assume $\delta _1(X) =1$. By assumption either $X$ is smooth or it has finitely many singular points. 

First assume that $X$ is smooth. By hypothesis, for a general $(p_0,p_1)\in X^2$ there is a reducible 
conic $C = L_0\cup L_1$ with $L_0$, $L_1$ lines $p_0\in L_1\setminus L_1\cap L_0$ and $p_1\in L_1\setminus L_1\cap L_0$. Thus $X$ is conic-connected by a reducible conic. Note that for a general $(p_0,p_1)$ the plane 
$\langle C\rangle$ is not contained in $X$ (we also have $X\cap \langle C\rangle =C$ scheme-theoretically). Since $\delta _1(X)=1$, for a general $q\in \langle C\rangle$ the reducible conic $C$ is an entry locus 
of $q$. Now apply \cite[Lemma 4.1.6]{r2}. 

Next assume $\mathrm{Sing}(X) \ne \emptyset$. By assumption $\mathrm{Sing}(X)$ is finite. For a general 
$p_0\in X$ let $\Ll _{p_0}$ denote the set of lines containing $p$ and $\Cc_{p_0} \subseteq \Ll _{p_0}$ 
be the set of all lines $L_0$ such that there is $p_1\in X$ with $L_0\cup L_1$ as the contact locus. 
Set $\Dd _{p_0}:= \cup _{L\in \Cc _{p_0}} L\subseteq X$. For any general $(p_0,p_1)\in X^2$
set $q_{p_0p_1}$ denote the point $L_0\cap L_1$, where $L_0\cup L_1$ is the contact locus of $\langle 
T_{p_0}X\cup T_{p_1}X\rangle$. Since $X$ is irreducible and the contact locus of a general $\langle T_{p_0}\cup T_{p_1}X\rangle$ is uniquely determined by $(p_0,p_1)$, for a general $p_0$ the set $\Ll _{p_0}$ is irreducible. Since each element of $\Ll_{p_0}$ is a line and $\mathrm{Sing}(X)$ is finite, to prove that $X$ is a cone it is sufficient to prove that a general $L\in \Ll _{p_0}$ contains a singular point of $X$.
This follows from a general deformation theory argument, as explained in the second half of the proof of \cite[Lemma 4.1.6]{r2}.

Since $L_i\cong \PP^1$ and $h^1(N_{L_i}) =1$, we get $h^1(N_{C|L_i}) =0$, hence 

\quad (b) Assume $\delta _1(X) \ge 2$. We do induction on the integer $\delta _1(X)$. Let $H\subset \PP^r$ 
be a general hyperplane. Set $Y:= X\cap H$. By Bertini's theorem $Y$ is an integral variety spanning $H$ 
and $\mathrm{Sing}(Y)$ has dimension at most $\dim \mathrm{Sing}(X)-1$, in particular $\mathrm{Sing}(Y) =\emptyset$ if $\dim \mathrm{Sing}(X) \le 0$. 

Since $\nu_1(Y)=\nu_1(X)-1$ and $\delta_1(Y) =\delta_1(X)-1$, as in step (b) of the proof of 
Lemma \ref{i1} we get by inductive assumption that $\dim \mathrm{Sing}(X) =\delta_1(X)-1$ and 
that $Y$ is a cone with $(\delta_1(X)-2)$-dimensional vertex and no other singular points.

\end{proof}

In order to prove Proposition \ref{t2} we need the following:

\begin{lemma}\label{t1}
Fix an integer $k \ge 1$. Let $X\subset \PP^r$ be an integral and non-degenerate variety such that 
$\mathrm{Sec}_k(X) \subset \PP^r$. Then $\nu_{k-1}(X)$ and $\nu_k(X)$ are well-defined. If 
$\nu_{k-1}(X)=\nu _k(X) >0$, then the generic tangential contact locus of $\mathrm{Sec}_{k-1}(X)$ 
and $\mathrm{Sec}_k(X)$ are of type II (i.e. reducible) and for a general $p\in X$ there is an integral 
$\nu_{k-1}(X)$-dimensional variety $A_p\subset X$ such that $p\in A_p$ and for a general
$(p_0,\dots ,p_{k+1})\in X^{k+2}$ we have $\Gamma _k(p_0,\dots ,p_k) = A_{p_0}\cup \cdots \cup A_{p_k}$ 
and $\Gamma_{k+1}(p_0,\dots ,p_{k+1}) =\Gamma _k(p_0,\dots ,p_k) \cup A_{p_{k+1}}$.
\end{lemma}

\begin{proof}
Fix a general $(p_0,\dots ,p_k)\in X^{k+2}$. Since $\mathrm{Sec}_k(X) \subset \PP^r$, Terracini's lemma 
(\cite[Corollary 1.11]{a}) gives that $M:= \langle T_{p_0}X\cup \cdots \cup T_{p_{k+1}}\rangle$
is a proper linear subspace of $\PP^r$. Thus both the contact locus $\Gamma _{k+1}(X)$ of $X$  
and the union $\Gamma _{k+1}(X)$ of its irreducible components containing one of the points 
$p_0,\dots ,p_{k+1}$ are well-defined. Thus for each $S\subset \{p_0,\dots ,p_{k+1}\}$ such that 
$\sharp (S)=k+1$ the linear space $W(S):=\langle \cup _{p\in S}T_pX\rangle$ is a proper linear subspace 
of $\PP^r$, hence the union $\Gamma _k(S)$ of the irreducible components of the contact locus of $W(S)$ 
and $X$ containing at least one point of $S$ is well-defined. For a general $p_i\in X$ we have 
$p_i\notin W(\{p_0,\dots ,p_{k+1}\}\setminus \{p_i\})$, hence $p_i\notin \Gamma _k(\{p_0,\dots ,p_{k+1} \} \setminus \{p_i\})$. By assumption $\dim \Gamma_{k+1}(X) = \dim \Gamma_k(S)$ and by the definition of tangential contact locus we have $\Gamma _{k+1}(X)\supseteq \Gamma_k(S)$ for all $S$. Since all general tangential contact loci are equidimensional, the claim follows.

\end{proof}

\begin{proof}[Proof of Proposition \ref{t2}:]
By \cite[Lemma 3.5]{bbc} the general tangential locus of order $t \in \{k-1,k \}$ is of type II 
(i.e. reducible) and its irreducible components are linearly independent $\nu_k(X)$-dimensional linear spaces. We get the existence of the linear spaces $L_p$, $p$ general in $X$, and that a general union 
of $k+1$ of them is linearly independent. Since $\mathrm{Sec}_{k+1}(X) \subset \PP^r$, $\nu_{k+1}(X)$ 
is well-defined and $\nu_{k+1}(X) \ge \nu_k(X) =\nu_{k-1}(X)$. It is easy to check that the general 
$(k+1)$-tangential contact locus is of type II (i.e. reducible).

\quad (a) Assume $\nu_{k+1}(X) = \nu_k(X)$. As in the proof of Lemma \ref{t1} we see that a general 
$(k+1)$-tangential contact locus is a union of $k+2$ general $L_p$'s. By \cite[Theorem 2.4]{cc1} 
applied to the integer $k+1$ we have $\dim \langle \cup_{P\in S} L_p \rangle = (k+2)(\nu_k(X)+1)-1 
-\delta_{k+1}(X)$.

\quad (b) Assume $\nu_{k+1}(X) >\nu _k(X)$. Since the $(k+1)$-tangential contact locus of type II (i.e. reducible) for a general $S\subset X$ with $\sharp(S) =k+2$, for each $p\in S$ there an irreducible 
$\nu_{k+1}(X)$-dimensional variety $A_p\subset X$ such that $p\in A_s$ and $\cup _{p\in S} A_s$ is 
contained in the contact locus of $X$ and the linear space $\langle _{p\in S} T_pX\rangle$. 
By \cite[part (ii) of Proposition 3.9]{cc2} we have $\dim \langle \cup_{p\in S} A_p\rangle = 
(k+2)(\nu_{k+1}(X) +1)-1 -\delta_{k+1}(X)$. Fix any $S'\subset S$ such that $\sharp(S') = k+1$ 
and define the linear spaces $L_p$ with respect to $S'$.

\end{proof}

Note that if (as in the statement of Proposition \ref{t2}) we have linear subspaces $L_p\subset X$, $p$ general in $X$, with positive dimension $w$ and such that a general union of $k+2$ of them spans a linear subspace of dimension $(k+2)(w+1)-\delta$ with $\delta >0$, then $\sum_{i=1}^{k+2} \delta _i(X) \ge \delta$. If we also have that a general union of $k+1$ $L_p$'s is linearly independent, then we have $\delta_{k+2}(X)\ge \delta$.


\begin{thebibliography}{99}

\bibitem{a} B. \r{A}dlandsvik, Joins and higher secant varieties. Math. Scand. 62 (1987), 213--222.

\bibitem{bbc} E. Ballico, A. Bernardi and L.  Chiantini, On the dimension of contact loci and the identifiability of tensors. Ark. Mat. 56 (2018), 265--283. 

\bibitem{bf} E. Ballico and C. Fontanari, Birational geometry of defective varieties, in: Projective Varieties with Unexpected Properties, 13--18. Walter de Gruyter, Berlin, 2005, available online at 
\url{https://arxiv.org/abs/math/0406320}.

\bibitem{bs} M. C. Beltrametti and A. J. Sommese, The adjunction theory of complex projective varieties. Walter de Gruyter, Berlin, 1995.

\bibitem{cc} L. Chiantini and C. Ciliberto, Weakly defective varieties. Trans. Amer. Math. Soc. 454 (2002), 151--178.

\bibitem{cc1} L. Chiantini and C. Ciliberto, On the concept of k-secant order of a variety. J.
London Math. Soc. (2) 73 (2006), 436--454.

\bibitem{cc2} L. Chiantini, and C. Ciliberto, On the dimension of secant varieties.
J. Europ. Math. Soc. 73 (2006), 436--454.

\bibitem{hh} R. Hartshorne and A. Hirschowitz, Smoothing algebraic space curves, in: Algebraic Geometry, 
Sitges 1983, 98--131. Lect. Notes in Math. 1124, Springer, Berlin, 1985.

\bibitem{r1} F. Russo, Varieties with quadratic entry locus, I. Math. Ann. 344 (2009), 597--617. 

\bibitem{r2} F. Russo, On the geometry of some special projective varieties. Lecture Notes of the Unione Matematica Italiana 18. Springer, Switzerland, 2016.

\bibitem{s} E. Sernesi, On the existence of certain families of curves. Invent. Math. 75 (1984), 25--57.

\bibitem{s0} E. Sernesi, Deformations of algebraic schemes. Springer, Berlin, 2006.

\bibitem{z} F. L. Zak, Tangents and secants of algebraic varieties, Transl. Math. Monog. 127. AMS, Providence, 1993.

\end{thebibliography}
\end{document}